\documentclass[a4,10pt]{article}

\usepackage{wrapfig}

\usepackage{lineno}

\usepackage{amsmath}
\usepackage{amsthm}
\usepackage{amssymb}
\usepackage{enumerate}
\usepackage[pdftex]{graphicx}
\usepackage{float}

\usepackage{url}
\usepackage[colorlinks=true, linkcolor=blue, citecolor=blue, urlcolor=blue]{hyperref}
\usepackage{cite}

\usepackage[skip=2pt]{caption}

\usepackage[top=36truemm,bottom=35truemm,left=32truemm,right=32truemm]{geometry}

\usepackage{titlesec}
\titleformat*{\section}{\large\bfseries}

\theoremstyle{definition}
\newtheorem{theorem}{Theorem}[section]

\newtheorem{definition}[theorem]{Definition}
\newtheorem{lemma}[theorem]{Lemma}

\newtheorem{proposition}[theorem]{Proposition}
\newtheorem{remark}[theorem]{Remark}
\numberwithin{equation}{section}

\allowdisplaybreaks

\begin{document}
  \title{\Large{Series expansions by generalized Bessel functions for functions\\ related to the lattice point problems for the $p$-circle}}
  \author{Masaya Kitajima}
  \date{}
  \maketitle
  \begin{abstract}
  For the lattice point problems of the $p$-circle, previous studies based on the generalized Bessel-function expansion introduced by E. Kr\"{a}tzel and the estimates of G. Kuba have solved the problems for sufficiently large $p$ (approximately $p>3$). However, for the range $0<p<2$, the conventional method based on this expansion encounters serious difficulties caused by singularities of the geometry of the $p$-circle, particularly in the presence of cusps when $0<p<1$. In this paper, motivated by the harmonic-analytic approach of S. Kuratsubo and E. Nakai to the Gauss circle problem, we introduce new generalized $p$-Bessel functions associated with $p$-radial symmetry and establish a series representation for generalized lattice point error functions related to the $p$-circle problem. More precisely, for $0<p\leq2$ satisfying $2/p\in\mathbb{N}$, we prove that these error functions admit pointwise expansions in terms of the generalized $p$-Bessel functions. This representation provides a new analytic framework for studying lattice point problems of the $p$-circles, especially in the previously inaccessible range $0<p\leq1$. 
  \end{abstract}\vspace{-15pt}
  
  \section{Introduction and main results}
  \hspace{13pt}For positive real numbers $p$ and $r$, let us focus on the lattice points inside the closed curve called the $p$-circle $\{x\in\mathbb{R}^{2}|\ |x_{1}|^{p}+|x_{2}|^{p}=r^{p}\}$ on the plane. This figure has several names, such as a Lam\'{e} curve and a superellipse, but in this paper we refer to them as a $p$-circle, in order to make the parameter $p$ explicit; for examples, see the left panel of Figure~\ref{fig}. 
  \par
  As shown in the right panel of Figure \ref{fig}, the mosaic can be approximated by a set of non-intersecting unit squares centered at each lattice point inside the closed curve, showing that the area of this mosaic coincides with the number of lattice points $R_{p}(r)$ in the $p$-circle. Then, $P_{p}(r):=R_{p}(r)-(2\Gamma(1/p)^{2}/p\Gamma(2/p))r^{2}$ serves as the error term of the approximation for the area, where the second term on the right-hand side is the area of the $p$-circle and $\Gamma$ denotes the gamma function. \par
  Now, let us consider the lattice point problems of the closed $p$-circle (namely, the problems of determining exponents $\alpha_{p}$ such that $P_{p}(r)=\mathcal{O}(r^{\alpha_{p}})$ and $P_{p}(r)=\Omega(r^{\alpha_{p}})$). Note that for the functions $f$ and $g$, $f(t)=\mathcal{O}(g(t))$ and $f(t)=\Omega(g(t))$ mean $\limsup_{t\to\infty}|\frac{f(t)}{g(t)}|<+\infty$ and $\limsup_{t\to\infty}|\frac{f(t)}{g(t)}|>0$, respectively. \par\vspace{5pt}
  In particular, in the case $p=2$, that is, when the closed curve is a circle, it is called the Gauss circle problem (see \cite{Gauss}), and in 1915 G.H. Hardy\cite{Hardy-1915} showed $P_{2}(r)=\Omega(r^{\frac{1}{2}}(\log r)^{4})$. Hardy and E. Landau further conjectured in \cite{Hardy-1917} that
  \begin{equation}\label{Hardy}
    P_{2}(r)=\mathcal{O}(r^{\frac{1}{2}+\varepsilon})\qquad \text{for any small }\varepsilon>0.
  \end{equation}
  This conjecture is motivated by the series representation of the error term $P_{2}$ using the Bessel function 
  \begin{equation}\label{HI}
    P_{2}(r)=r\sum_{n=1}^{\infty}\frac{R(n)}{\sqrt{n}}J_{1}(2\pi\sqrt{n}r)\quad\text{for }r>0\quad \text{with }R(n):=\#\{m\in\mathbb{Z}^{2}|\ |m|^{2}=n\},
  \end{equation}
  which is well known as \textit{Hardy's identity}\cite{Hardy-1915}, and the representation due to A.~Ivic\cite{Ivic} 
  \begin{figure}[t]
    \centering
    \includegraphics[width=0.8\linewidth]{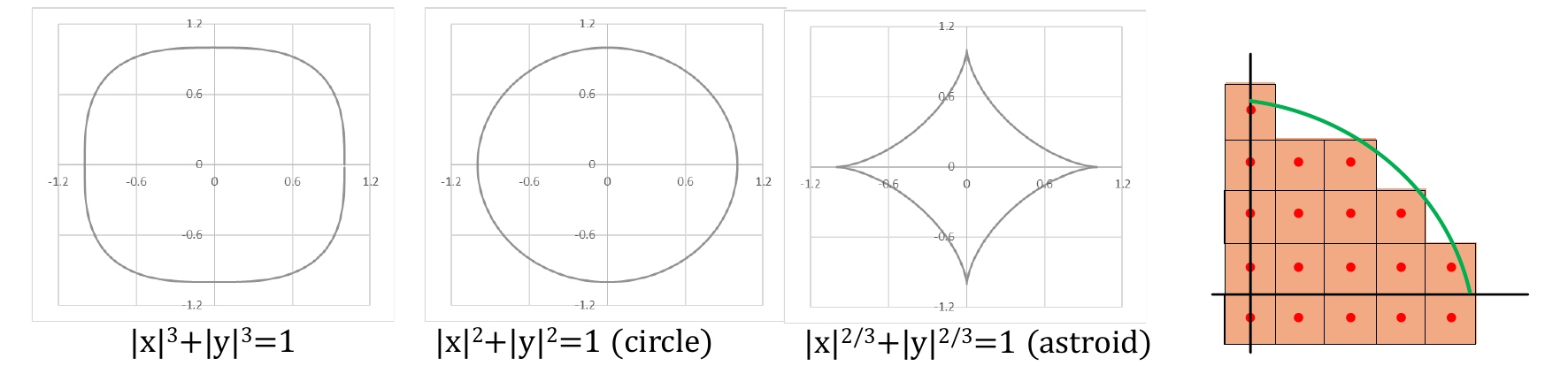}
    \caption{Examples of the $p$-circle and the approximation by unit squares.}
    \label{fig}
  \end{figure}
  \begin{align*}
    P_{2}(r)=\frac{1}{\pi}r^{\frac{1}{2}}&\sum_{1\leq n<M}\frac{R(n)}{n^{\frac{3}{4}}}\cos(2\pi\sqrt{n}r-3\pi/4)\ +\ (\text{remainder terms})\\
    & \text{for }r>1,\ r^{2}\not\in\mathbb{N},\text{ and }r^{2}\leq M \leq r^{2A},\text{  with some fixed constant }A>1
  \end{align*}
  guarantees, in a sense, the validity of the conjecture. \par
  Since then, the $\mathcal{O}$-estimates of the error term $P_{2}$ have been improved, and the current best result shows that the claim (\ref{Hardy}) holds for any $\varepsilon>0.128989\cdots$ by X. Li and X. Yang \cite{XX} in 2023, but the conjecture itself remains open.\par\vspace{5pt}
  
  On the other hand, in the cases $p>2$, by considering the second main term \cite[(3.55)]{Kratzel}
  \begin{equation}\label{smt}
    \Psi(r;p):=8\sqrt{\pi}\Gamma(1+1/p)\sum_{n=1}^{\infty}\left(\frac{r}{\pi n}\right)J_{\frac{2}{p}}^{(p)}(2\pi nr)\qquad\text{for }r>0
  \end{equation}
  given by the series representation using E. Kr\"{a}tzel's generalized Bessel functions \cite[Definition 3.3]{Kratzel} with $p\geq1$
  \begin{equation}\label{Kratzel-J}
    J_{\nu}^{(p)}(r):=\frac{2}{\sqrt{\pi}\Gamma(\nu+1-\frac{1}{p})}\left(\frac{r}{2}\right)^{\frac{p\nu}{2}}\int_{0}^{1}(1-t^{p})^{\nu-\frac{1}{p}}\cos rt\ dt\qquad\text{for }r>0,\ \nu>\frac{1}{p}-1,
  \end{equation}  
  the representation \cite[(3.57)]{Kratzel} 
   \begin{equation}
    P_{p}(r)=\Psi(r;p)+\Delta(r;p).
  \end{equation}
  with $\Delta(r;p)(=o(r^{\frac{2}{3}}))$ as the remainder term is obtained.
  Note that $f(t)=o(g(t))$ for the functions $f$ and $g$ means $\lim_{t\to\infty}|\frac{f(t)}{g(t)}|=0$. In these cases, combining the known estimate $\Psi(r;p)\asymp r^{1-\frac{1}{p}}$ which follows from the asymptotic expansion of $\Psi(r;p)$, immediately yields Theorem~\ref{upper}. Here and throughout the paper, for two positive functions $f$ and $g$, we write $f(t)\asymp g(t)$
if $f(t)=\mathcal{O}(g(t))$ and $f(t)=\Omega(g(t))$.
  \begin{theorem}[\itshape{\hspace{-0.5pt}\cite[Theorem 3.17 A]{Kratzel}}]\label{upper}
    \itshape{Let $p>2$. If there exists $\alpha_{p}<1-\frac{1}{p}$ such that $\Delta(r;p)=\mathcal{O}(r^{\alpha_{p}})$, then $P_{p}(r)\asymp r^{1-\frac{1}{p}}$ holds.}
  \end{theorem}
  Furthermore, applying the latest result $\Delta(r;p)=\mathcal{O}(r^{\frac{46}{73}}(\log r)^{\frac{315}{146}})$ by G. Kuba\cite{Kuba} in 1993, it follows that the lattice point problems of the $p$-circle for $p>73/27$ have been solved. \par\vspace{5pt}
  Thus, previous research on the lattice point problems of the $p$-circle in the sense of matching $\mathcal{O}$ and $\Omega$-estimates, has primarily focused on figures for which $p$ is greater than or equal to $2$. On the other hand, in another sense, the lattice point problems have been widely studied, including the cases $0<p<2$. For example, for a family of $s$ of lattice points in the $p$-ellipse $\{x\in\mathbb{R}^{2}|\ |sx_{1}|^{p}+|\frac{x_{2}}{s}|^{p}=r^{p}\}\ (s>0)$, the general form of the $p$-circle, there are problems finding $s$ such that the number of lattice points is the largest. For these problems, it was shown by R.S. Laugesen and S. Ariturk\cite{Laugesen1} in 2017, and by Laugesen and S. Liu\cite{Laugesen2} in 2018, that the closer the $p$-ellipse is to the $p$-circle (that is, as $s\to1$), the more lattice points it contains in the cases $0<p<1$ or $p>1$ (the case $p=1$ is still open). \par\vspace{3pt}
  Therefore, in addition to the cases $p\geq2$, which has been the subject of previous studies, it is also meaningful to study the lattice point problems of the $p$-circle (in the original sense) in the cases $0<p<2$. As mentioned above, however, the attempt to lower $p$ for the solved problem has been stagnant for more than 30 years. Let us recall that the Kr\"{a}tzel's method based on the second main term, as formulated in Theorem \ref{upper}, is conventional in the cases $p\geq2$. This method deals with the second-order derivative $x''_{2}(x_{1})$ for $x_{2}>0$ with $r$ fixed. In the cases $0<p<2$, however, singularities arise in this derivative. Hence, it would be difficult to approach the problem along this method. \par\vspace{3pt}
  From a geometric viewpoint, it is evident that the analysis becomes difficult, particularly in the case $0<p<1$, since the curve possesses cusps at which the curvature becomes unbounded. Thus, we need another method based on functions for which such singularities do not appear, and attempt to conduct a study based on series expansion using new generalized Bessel functions, especially for the unsolved lattice point problems of the $p$-circle for $0<p<2$. \par\vspace{5pt}
  This attempt is inspired by the harmonic-analytic approach to Hardy's conjecture for the case $p=2$, introduced by S.~Kuratsubo and E.~Nakai \cite{Kuratsubo-2022}. In their method, the series representation
  \begin{equation}\label{D-J}
      D_{\beta}(s:x)-\mathcal{D}_{\beta}(s:x)=s^{\beta+1}2^{\beta+1}\pi\sum_{ n\in\mathbb{Z}^{2}\setminus\{0\}}\frac{J_{\beta+1}(2\pi\sqrt{s}|x-n|)}{(2\pi\sqrt{s}|x-n|)^{\beta+1}}\qquad\text{ if }\beta>\frac{1}{2}
  \end{equation}
  using the Bessel functions for the functions $D_{\beta}(s:x)-\mathcal{D}_{\beta}(s:x)$, which generalize the error term $P_{2}$ by variables, played an important role \cite[(2.6)]{Kuratsubo-2022}. Note that for $\beta>-1,\ s>0,\ x\in\mathbb{R}^{2}$, we have
    \begin{equation}\label{D}
    D_{\beta}(s:x):=\frac{1}{\Gamma(\beta+1)}\sum_{|m|^{2}<s}(s-|m|^{2})^{\beta}
    e^{2\pi ix\cdot m},\quad 
    \mathcal{D}_{\beta}(s:x):=\frac{1}{\Gamma(\beta+1)}\int_{|\xi|^{2}<s}
    (s-|\xi|^{2})^{\beta}e^{2\pi ix\cdot \xi}d\xi.
    \end{equation}
    In particular, $D_{0}(r^{2}:0)-\mathcal{D}_{0}(r^{2}:0)=P_{2}(r)$. Then, by using this series representation (\ref{D-J}), they obtained the $\mathcal{O}$-estimate for $D_{0}(s:x)-\mathcal{D}_{0}(s:x)$ as $s\to\infty$, uniformly in $x$.

   \begin{proposition}[\itshape{\cite[Lemma 5.1(Special cases)]{Kuratsubo-2022}}]\label{prop2022}
    \itshape{Under the above definitions of the functions,}\vspace{-2pt}
    \begin{equation*}\vspace{-2pt}
      D_{0}(s:x)-\mathcal{D}_{0}(s:x)=\mathcal{O}(s^{\frac{1}{3}})
    \end{equation*}
    \itshape{uniformly with respect to $x\in(-\frac{1}{2},\frac{1}{2}]^{2}$. In particular, 
    $P_{2}(r)=\mathcal{O}(r^{\frac{2}{3}})$ holds.}
    \end{proposition}
  This $\mathcal{O}$-estimate of $P_{2}$ is a relatively crude evaluation by G.F. Vorono\"{i}\cite{Voronoi} in 1904 and W. Sierpi\'{n}ski\cite{Sierpinski} in 1906. However, at present, even such a crude estimate has not been obtained for the cases $0<p<2$. Therefore, if Proposition \ref{prop2022} can be generalized for smaller $p$ by following Kuratsubo and Nakai's method of using the series representation (\ref{D-J}), it is expected that the error term can be estimated while avoiding problems caused by curvature and singularities.\par\vspace{5pt}
  Thus, in this paper, we need to consider the $p$-circle lattice points error functions generalized by variables $D_{\beta}^{[p]}-\mathcal{D}_{\beta}^{[p]}$, which generalizes the Kuratsubo and Nakai's functions (\ref{D}) for $p$ to
  \begin{equation}\label{Dp}
    D_{\beta}^{[p]}(s:x):=\frac{1}{\Gamma(\beta+1)}\sum_{|m|_{p}^{p}<s}(s-|m|_{p}^{p})^{\beta}e^{2\pi ix\cdot m},\quad 
    \mathcal{D}_{\beta}^{[p]}(s:x):=\frac{1}{\Gamma(\beta+1)}\int_{|\xi|_{p}^{p}<s}
    (s-|\xi|_{p}^{p})^{\beta}e^{2\pi ix\cdot \xi}d\xi,
  \end{equation}
  where $|\cdot|_{p}$ is the $p$-norm, namely, $|\xi|_{p}=(|\xi_{1}|^{p}+|\xi_{2}|^{p})^{\frac{1}{p}}$ for $\xi=(\xi_{1},\xi_{2})\in\mathbb{R}^{2}$. \par
  Then, we generalize the Bessel function of order zero for $p>0$ to
  \begin{equation}\label{Bessel_0}
    J_{0}^{[p]}(x):=\frac{1}{\Gamma(\frac{1}{p})^{2}}\left(\frac{2}{p}
    \right)^{2}\int_{0}^{1}\cos(x_{1}t^{\frac{1}{p}})\cos(x_{2}(1-t)^{\frac{1}{p}})
    t^{\frac{1}{p}-1}(1-t)^{\frac{1}{p}-1}dt\qquad\text{for }x\in\mathbb{R}^{2}.
  \end{equation}
  Note that this is different from the function (\ref{Kratzel-J}) by Kr\"{a}tzel and details are given in the next section. \par
  We emphasize that this paper, together with our previous works\cite{K2, K3, K4}, is concerned with small values of $p$ in the range $0<p<2$, in particular when $2/p$ is a natural number (owing to its shape, we call the corresponding closed curve an \textit{astroid-type p-circle}).
  This assumption is a natural requirement for treating analytically difficult problems and, in particular, for establishing several properties of the $p$-Bessel functions that play crucial roles in this work. Specifically, it guarantees the validity of the interchange of summation and integration in the derivation of the series expansion (Proposition \ref{series}), the generalized uniform convergence of the series representation of the \textit{$p$-Bessel functions of positive order} in \cite{K3} as
  \begin{equation}\label{Bessel_omega}
    J_{\omega}^{[p]}(x):=\frac{|x|_{p}^{\omega}}{p^{\omega-1}\Gamma(\omega)}
    \int_{0}^{1}J_{0}^{[p]}(\tau x)\tau(1-\tau^{p})^{\omega-1}
    d\tau\qquad\text{for }x\in\mathbb{R}^{2},\ \omega>0,
  \end{equation}
  and the nondegeneracy of the phase function appearing in the oscillatory integral representation of $J_{\omega}^{[p]}(x)$ in \cite{K2, K4}. \par
  Then, by using these $p$-Bessel functions for $p$ such that $2/p\in\mathbb{N}$, the following theorem (the main result of this paper) holds, which provides a generalization of Kuratsubo and Nakai's series expansion (\ref{D-J}). Moreover, it may be regarded as the initial step of our on-going research project (see Section 4).
  \begin{theorem}\label{thm}
  \itshape{Let $0<p\leq2$ satisfy $2/p\in\mathbb{N}$. If $\beta>-1$ is such that $\mathcal{D}_{\beta}^{[p]}(1:x)$ is integrable on $\mathbb{R}^{2}$, then the following series representation holds pointwise on $\mathbb{R}^{2}$:}\vspace{-3pt}
  \begin{equation*}
    D_{\beta}^{[p]}(s:x)-\mathcal{D}_{\beta}^{[p]}(s:x)=s^{\beta+\frac{2}{p}}p^{\beta+1}
    \Gamma(1/p)^{2}\sum_{n\in\mathbb{Z}^{2}\setminus\{0\}}\frac{J_{\beta+1}^{[p]}
    (2\pi\sqrt[p]{s}(x-n))}{(2\pi \sqrt[p]{s}|x-n|_{p})^{\beta+1}}
    \qquad\text{for }s>0,\ x\in\mathbb{R}^{2}.
  \end{equation*}
  Furthermore, under this assumption, the series converges absolutely for $x\in\mathbb{T}^{2}:=(-\frac{1}{2},\frac{1}{2}]^{2}$.
  \end{theorem}
  \begin{remark}
  A closer look at the arguments in Section 3 reveals that the series representation in this theorem is developed under the integrability assumptions on $\beta$ that guarantee pointwise convergence on $\mathbb{R}^{2}$ and absolute convergence on $\mathbb{T}^{2}$. On the other hand, as discussed in Section 4, in order to make the assumption on $\beta$ in the theorem explicit, the essential issue is the asymptotic behavior of the $p$-Bessel functions of arbitrary order. However, except for the zeroth-order case, such asymptotic behavior remains unknown. (We emphasize that, in the classical case $p=2$, the standard asymptotic estimates for the classical Bessel functions are available, and hence the condition $\beta>1/2$ is equivalent to the corresponding integrability condition.) In Section 4, we review the known asymptotic behavior of the $p$-Bessel functions and discuss the problem of determining the precise range of $\beta$ in the integrability assumption. In particular, we explain why the asymptotic behavior along the coordinate axes alone is insufficient to determine the optimal range, leaving the derivation of uniform asymptotic estimates for general orders to future work. Finally, in the case $\beta=0$ and $x=0$, the conditional convergence itself has already been obtained independently of the present paper, without imposing any integrability condition, as the main result of \cite{K3} (a generalized Hardy identity). However, since the error term $P_{p}$ is oscillatory, one cannot expect absolute convergence of the corresponding series representation.
  \end{remark}
  This paper is organized as follows: in Section 2, the analytic background of the $p$-Bessel functions $J_{\omega}^{[p]}$ is described and definitions (\ref{Bessel_0}) and (\ref{Bessel_omega}) are given, and the proof of Theorem \ref{thm}, the main result of this paper, is given in Section 3. Finally, in Section 4, we conclude this paper by discussing future attempts to specify $\beta$ in the assumption of Theorem \ref{thm}.\vspace{7pt}
  
  \section{Definitions and properties of $p$-Bessel functions}
  \hspace{13pt}As a point of departure, recall that a function $F$ on $\mathbb{R}^{2}$ is spherically symmetric if there exists a function $\phi$ on nonnegative real numbers such that $F(x)=\phi(|x|)$ holds for any $x\in\mathbb{R}^{2}$, and that if such a function $F$ is integrable on $\mathbb{R}^{2}$, its Fourier transform can be written as
  \begin{equation}\label{2rad-Fourier}
    \hat{F}(\xi)=2\pi\int_{0}^{\infty}J_{0}(2\pi r|\xi|)\phi(r)rdr
    \quad\text{for }\xi\in\mathbb{R}^{2}
  \end{equation}
  via integration against the Bessel function $J_{0}$ (Hankel transform of order zero; \cite{Gaskill} or \cite{Stein-Shakarchi}). \par
  Based on this result, we now introduce a $p$-analogue of the Bessel function of order zero by considering the notion of $p$-radial, which extends spherical symmetry, and by deriving the corresponding integral representation associated with this property.

  \begin{definition}[\itshape{$p$-radial}]\label{p-rad}
    \itshape{Let $p$ be a positive real number and $|x|_{p}:=(|x_{1}|^{p}+|x_{2}|^{p})^{\frac{1}{p}}$ be the $p$-norm. Then, for a function $F$ on $\mathbb{R}^{2}$, $F$ is said to be $p$-radial if there exists a function $\phi$ on nonnegative real numbers such that $F(x)=\phi(|x|_{p})$ holds for any $x\in\mathbb{R}^{2}$. }
  \end{definition}\vspace{7pt}
  First, to obtain the desired integral representation, we consider the derivation of the transformation (\ref{2rad-Fourier}) for a spherically symmetric integrable function $F$ on $\mathbb{R}^{2}$. By employing the well-known fact that the Fourier transform of a spherically symmetric integrable function remains spherically symmetric, the representation follows from a straightforward change of variables with a suitable choice of $\xi \in \mathbb{R}^{2}$.\par
In contrast, when spherical symmetry is generalized to the $p$-radial setting, an analogous property is no longer available. Consequently, the integral representation in full generality cannot be derived by the simple argument above, and therefore must instead be established by the following approach.
  \par\vspace{12pt}
  To compute the Fourier transform of an integrable and $p$-radial function $F$ on $\mathbb{R}^{2}$, the integral range is divided into four quadrants in the plane $\mathbb{R}^{2}$ and the change of variables 
  \begin{equation*}
    x_{1}=\mathrm{sgn}(\cos\theta)\ r|\cos\theta|^{\frac{2}{p}},\quad
    x_{2}=\mathrm{sgn}(\sin\theta)\ r|\sin\theta|^{\frac{2}{p}},\quad
    \text{for }r\geq0,\ 0\leq\theta<2\pi
  \end{equation*}
  is applied (note that $\mathrm{sgn}$ is the sign function). Denote by $\mathcal{J}_{[j]}(r,\theta)$ the Jacobian in the $j$-th quadrant. Then, by transforming the integral as follows, we obtain the desired integral representation together with the corresponding $p$-Bessel function of order zero.
  \begin{align}\label{osc4}
    \hat{F}(\xi)&=\int_{0}^{\infty}\phi(r)\Bigl(\int_{0}^{\frac{\pi}{2}}e^{-2\pi ir
    (\xi_{1}\cos^{\frac{2}{p}}\theta+\xi_{2}\sin^{\frac{2}{p}}\theta)}\mathcal{J}_{[1]}
    (r,\theta)d\theta+\int_{\frac{\pi}{2}}^{\pi}e^{-2\pi ir(-\xi_{1}(-\cos
    \theta)^{\frac{2}{p}}+\xi_{2}\sin^{\frac{2}{p}}\theta)}\mathcal{J}_{[2]}(r,\theta)
    d\theta\notag\\
    &\hspace{10pt}+\int_{\pi}^{\frac{3}{2}\pi}e^{-2\pi ir(-\xi_{1}(-\cos\theta)^{\frac{2}
    {p}}-\xi_{2}(-\sin\theta)^{\frac{2}{p}})}\mathcal{J}_{[3]}(r,\theta)d\theta+
    \int_{\frac{3}{2}\pi}^{2\pi}e^{-2\pi ir(\xi_{1}\cos^{\frac{2}{p}}\theta-\xi_{2}(-\sin
    \theta)^{\frac{2}{p}})}\mathcal{J}_{[4]}(r,\theta)d\theta\Bigr) dr\notag\\
    &=\frac{2}{p}\int_{0}^{\infty}\phi(r)r\int_{0}^{\frac{\pi}{2}}\bigl(e^{-2\pi ir
    (\xi_{1}\cos^{\frac{2}{p}}\theta+\xi_{2}\sin^{\frac{2}{p}}\theta)}+e^{2\pi ir(\xi_{1}
    \sin^{\frac{2}{p}}\theta-\xi_{2}\cos^{\frac{2}{p}}\theta)}+e^{2\pi ir(\xi_{1}
    \cos^{\frac{2}{p}}\theta+\xi_{2}\sin^{\frac{2}{p}}\theta)}\notag\\
    &\hspace{83pt}+e^{-2\pi ir(\xi_{1}\sin^{\frac{2}{p}}\theta-\xi_{2}\cos^{\frac{2}{p}}
    \theta)}\bigr)(\cos\theta\sin\theta)^{\frac{2}{p}-1} d\theta dr\notag\\
    &=\frac{2}{p}\int_{0}^{\infty}\phi(r)r\int_{0}^{\frac{\pi}{2}}\bigl(\cos[2\pi r
    (\xi_{1}\cos^{\frac{2}{p}}\theta+\xi_{2}\sin^{\frac{2}{p}}\theta)]+\cos[2\pi 
    r(\xi_{1}\sin^{\frac{2}{p}}\theta-\xi_{2}\cos^{\frac{2}{p}}\theta)]\bigr)\notag\\
    &\hspace{250pt}\times(\cos^{2}\theta)^{\frac{1}{p}-1}(\sin^{2}\theta)^{\frac{1}{p}-1}
    2\sin\theta\cos\theta d\theta dr\notag\\
    &=\frac{2}{p}\int_{0}^{\infty}\Bigl(\int_{0}^{1}\bigl(\cos[2\pi r(\xi_{1}
    t^{\frac{1}{p}}+\xi_{2}(1-t)^{\frac{1}{p}})]+\cos[2\pi r(\xi_{1}t^{\frac{1}{p}}-
    \xi_{2}(1-t)^{\frac{1}{p}})]\bigr)(1-t)^{\frac{1}{p}-1}t^{\frac{1}{p}-1}dt\Bigr)\notag\\
    &\hspace{370pt}\times\phi(r)rdr\notag\\
    &=p\Gamma\left(\frac{1}{p}\right)^{2}\int_{0}^{\infty}\Bigl(\frac{1}{\Gamma(\frac{1}{p})^{2}}
    \left(\frac{2}{p}\right)^{2}\int_{0}^{1}\cos(2\pi r\xi_{1}t^{\frac{1}{p}})
    \cos(2\pi r\xi_{2}(1-t)^{\frac{1}{p}})t^{\frac{1}{p}-1}(1-t)^{\frac{1}{p}-1}dt\Bigr)
    \phi(r)rdr,\notag
  \end{align}
  \begin{equation}\label{prad-Fourier}
    \hat{F}(\xi)=p\Gamma(1/p)^{2}\int_{0}^{\infty}J_{0}^{[p]}(2\pi r\xi)\phi(r)
    rdr\qquad \text{for }\xi\in\mathbb{R}^{2},
  \end{equation}
  \vspace{-6pt}
  \begin{equation}\label{p-Bessel}
    \text{where }J_{0}^{[p]}(x):=\frac{1}{\Gamma(\frac{1}{p})^{2}}\left(\frac{2}{p}
    \right)^{2}\int_{0}^{1}\cos(x_{1}t^{\frac{1}{p}})\cos(x_{2}(1-t)^{\frac{1}{p}})
    t^{\frac{1}{p}-1}(1-t)^{\frac{1}{p}-1}dt\qquad\text{for }x\in\mathbb{R}^{2}.
  \end{equation}
  \begin{remark}
  The function $J_{0}^{[p]}$ is a generalization of the Bessel function $J_{\omega}$ at $\omega=0$. Note that the domains of definition are different. In fact, for $p=2$ and $x=(r\cos\tau,r\sin\tau)\ (r>0,\ 0\leq\tau<2\pi)$, it matches Poisson's integral representation of the Bessel function \cite[p47; (1)]{Watson}, as 
  \begin{align}
    J_{0}^{[2]}(x)&=\frac{1}{\pi}\int_{0}^{1}\cos(r\sqrt{t}\cos\tau)\cos(r
    \sqrt{1-t}\sin\tau)t^{-\frac{1}{2}}(1-t)^{-\frac{1}{2}}dt\notag\\
    &=-\frac{2}{\pi}\int_{\frac{\pi}{2}}^{0}\cos(r\cos\tau\cos\theta)\cos(r\sin\tau\sin
    \theta)d\theta\notag\\
    &=\frac{1}{\pi}\int_{0}^{\frac{\pi}{2}}(\cos(r(\cos\tau\cos\theta+\sin\tau
    \sin\theta))+\cos(r(\cos\tau\cos\theta-\sin\tau\sin\theta)))d\theta\notag\\
    &=\frac{1}{\pi}\int_{0}^{\frac{\pi}{2}}(\cos(r(\cos(\tau-\theta)))+\cos(r\cos(\tau
    +\theta)))d\theta\notag\\
    &=\frac{1}{\pi}\int_{-\frac{\pi}{2}+\tau}^{\frac{\pi}{2}+\tau}\cos(r\cos\theta)
    d\theta\notag\\
    &=\frac{1}{\pi}\int_{0}^{\pi}\cos(r\cos\theta)d\theta=J_{0}(|x|),\label{p=2}
  \end{align}
  and in particular $J_{0}^{[2]}$ has spherical symmetry. Thus, it is clear that this integral representation (\ref{prad-Fourier}) is a generalization of the Hankel transform of order zero (\ref{2rad-Fourier}).
  \end{remark}\vspace{5pt}
  Hereafter, let $p$ satisfy $2/p\in\mathbb{N}$.\par
  Furthermore, for $n\in\mathbb{N}_{0}^{2},\ x\in\mathbb{R}^{2},\ \xi\in\mathbb{R}^{2}_{>0}$, the following multi-index notation is used to denote the series expansions of $J_{0}^{[p]}$. 
  \begin{equation*}
    |n|':=n_{1}+n_{2},\quad x^{n}:= x_{1}^{n_{1}}\cdot x_{2}^{n_{2}},\quad  
    n!:=(n_{1})!(n_{2})!,\quad\Gamma(\xi):=\Gamma(\xi_{1})\Gamma(\xi_{2}).
  \end{equation*}
  \begin{proposition}[\itshape{The series expansion of $J_{0}^{[p]}$}]\label{series}
    \itshape{}
    \begin{equation*}
      J_{0}^{[p]}(x)=\frac{1}{\Gamma(\frac{1}{p})^{2}}\left(\frac{2}{p}
      \right)^{2}\sum_{k=0}^{\infty}\frac{(-1)^{k}}{\Gamma(\frac{2}{p}(k+1))}
      \sum_{m\in\mathbb{N}_{0}^{2}\ |m|'=k}\frac{\Gamma(\frac{2m+1}{p})}{(2m)!}x^{2m}
      \qquad\text{for }x\in\mathbb{R}^{2}.
    \end{equation*}
  \end{proposition}
  \begin{proof}
    For $x\in\mathbb{R}^{2}$, expanding the cosine functions into their power series and applying the Cauchy product formula to regroup the resulting series, followed by term-by-term integration, yields
    \begin{align*}
      \int_{0}^{1}\cos(x_{1}&t^{\frac{1}{p}})\cos(x_{2}(1-t)^{\frac{1}{p}})
      t^{\frac{1}{p}-1}(1-t)^{\frac{1}{p}-1}dt\\
      &=\int_{0}^{1}\left(\sum_{m_{1}=0}^{\infty}\frac{(-1)^{m_{1}}}{(2m_{1})!}(x_{1}
      t^{\frac{1}{p}})^{2m_{1}}\right)\left(\sum_{m_{2}=0}^{\infty}
      \frac{(-1)^{m_{2}}}{(2m_{2})!}(x_{2}(1-t)^{\frac{1}{p}})^{2m_{2}}\right)
      t^{\frac{1}{p}-1}(1-t)^{\frac{1}{p}-1}dt\\
      &=\int_{0}^{1}\left(\sum_{k=0}^{\infty}\sum_{|m|'=k}
      \frac{(-1)^{k}(x_{1}t^{\frac{1}{p}})^{2m_{1}}
      (x_{2}(1-t)^{\frac{1}{p}})^{2m_{2}}}{(2m_{1})!(2m_{2})!}\right)
      t^{\frac{1}{p}-1}(1-t)^{\frac{1}{p}-1}dt\\
      &=\sum_{k=0}^{\infty}\sum_{|m|'=k}
      \frac{(-1)^{k}x_{1}^{2m_{1}}x_{2}^{2m_{2}}}{(2m_{1})!(2m_{2})!}
      \int_{0}^{1}t^{\frac{1}{p}(2m_{1}+1)-1}(1-t)^{\frac{1}{p}(2m_{2}+1)-1}dt\\
      &=\sum_{k=0}^{\infty}(-1)^{k}\sum_{|m|'=k}\frac{x^{2m}}{(2m)!}
      \frac{\Gamma(\frac{2m_{1}+1}{p})\Gamma(\frac{2m_{2}+1}{p})}{\Gamma(\frac{2}{p}(k+1))}.
    \end{align*}
    Note that the interchange of the infinite sum and the integral is justified under the assumption that $2/p$ is a positive integer, and the last equality follows from the beta-gamma relation.
  \end{proof}
  \begin{remark}
  This type of generalization of $J_{0}$ for $p$ has already been studied by D.St.P. Richards\cite{Richards-1985, Richards-1986} in the multidimensional setting in terms of the functions\vspace{-5pt}
  \begin{equation*}
    J_{n,p}(x):=\int_{|\xi|_{p}=1}e^{ix\cdot \xi}\omega(\xi)\quad\text{for }x\in\mathbb{R}^{n},\ 
    \text{where }
    \omega(\xi):=\sum_{j=1}^{n}(-1)^{j-1}\xi_{j}d\xi_{1}\cdots d\xi_{j-1}d\xi_{j+1}
    \cdots d\xi_{n},
  \end{equation*}
  whose two-dimensional case satisfies $J_{2,p}=p\Gamma(1/p)^{2}J_{0}^{[p]}$, although the construction is different. Richards developed a general framework for positive definite symmetric functions on finite-dimensional spaces, with emphasis on representations via Radon transforms and characterizations of such functions arising from harmonic analysis. The counterparts of the equality for the general $n$ cases (\ref{prad-Fourier}) have also been obtained \cite[(2.3)]{Richards-1985}. In addition, in 1986, Richards derived series expansions of $J_{n,p}$ by using Stokes' theorem \cite[Theorem 4.3]{Richards-1986}. The previous proof shows that this result, especially for $n=2$, corresponds to Proposition \ref{series} and can be derived in two dimensions by essentially the same argument. Next, in 2008, W. zu Castell presented another proof of the series expansions by using a recurrence formula (of $J_{n,p}$ with respect to the dimension) he had found \cite[Theorems 3-4]{Castell}.
  \end{remark}
  To initiate the study of the lattice point problem for the $p$-circle, it is necessary to establish a framework in the theory of special functions for the generalized Bessel function associated with $p$. In the latter part of this section, we introduce orders for $J_{0}^{[p]}$ ($\backsimeq_{p}J_{2,p}$).\par\vspace{5pt}
  Let $B(s,t)$ be the beta function. Motivated by the well-known equality for the Bessel functions \cite[Lemma 4.13]{Stein-1971}
  \begin{equation*}
    J_{\alpha+\beta+1}(t)=\frac{t^{\beta+1}}{2^{\beta}\Gamma(\beta+1)}\int_{0}^{1}
    J_{\alpha}(ts)s^{\alpha+1}(1-s^{2})^{\beta}ds\qquad\text{for }\alpha>-\frac{1}{2},\ \beta>-1,\ t>0
  \end{equation*}
  and the transformation for $0<p\leq2$ such that $2/p\in\mathbb{N}$, $\beta>-1$
  \begin{align*}
    \int_{0}^{1}&\Bigl(\int_{0}^{1}\cos(\tau x_{1}t^{\frac{1}{p}})\cos(\tau x_{2}
    (1-t)^{\frac{1}{p}})t^{\frac{1}{p}-1}(1-t)^{\frac{1}{p}-1}dt\Bigr)\tau
    (1-\tau^{p})^{\beta}d\tau\\
    &=\int_{0}^{1}\int_{0}^{1}\Bigl(\sum_{k=0}^{\infty}\sum_{|m|'=k}\frac{(-1)^{k}
    (\tau x_{1}t^{\frac{1}{p}})^{2m_{1}}(\tau x_{2}(1-t)^{\frac{1}{p}})^{2m_{2}}}
    {(2m)!}\Bigr)t^{\frac{1}{p}-1}(1-t)^{\frac{1}{p}-1}\tau(1-\tau^{p})^{\beta}dtd\tau\\
    &=\sum_{k=0}^{\infty}(-1)^{k}\sum_{|m|'=k}\frac{x^{2m}}{(2m)!}B\Bigl(
    \frac{2m_{1}+1}{p},\frac{2m_{2}+1}{p}\Bigr)\cdot\frac{1}{p}\cdot
    B\Bigl(\frac{2(k+1)}{p},\beta+1\Bigr)\\
    &=\frac{\Gamma(\beta+1)}{p}\sum_{k=0}^{\infty}\frac{(-1)^{k}}{\Gamma(\frac{2}{p}(k+1)+\beta+1)}\sum_{|m|'=k}\frac{\Gamma(\frac{2m_{1}+1}{p})\Gamma(\frac{2m_{2}+1}{p})}
    {(2m)!}x^{2m},
  \end{align*}
  \begin{equation}\label{omega_J}
    \frac{|x|_{p}^{\beta+1}}{p^{\beta}\Gamma(\beta+1)}\int_{0}^{1}J_{0}^{[p]}
    (\tau x)\tau(1-\tau^{p})^{\beta}d\tau
    =\frac{|x|_{p}^{\beta+1}(\frac{2}{p})^{2}}{p^{\beta+1}\Gamma
    (\frac{1}{p})^{2}}\sum_{k=0}^{\infty}\frac{(-1)^{k}}{\Gamma(\frac{2}{p}(k+1)+\beta+1)}
    \sum_{m\in\mathbb{N}_{0}^{2}\ |m|'=k}\frac{\Gamma(\frac{2m+1}{p})}{(2m)!}x^{2m}
  \end{equation}
  hold. Thus, based on the above, the definition of $p$-Bessel functions of order non-negative $J_{\omega}^{[p]}$ and their series expansions are given as follows.
  \begin{definition}\label{I-Om}\itshape{
    \begin{equation*}
      J_{\omega}^{[p]}(x):=
      \begin{cases}
        \displaystyle \frac{|x|_{p}^{\omega}}{p^{\omega-1}\Gamma(\omega)}\int_{0}^{1}
        J_{0}^{[p]}(\tau x)\tau(1-\tau^{p})^{\omega-1}
        d\tau\qquad\text{if }\omega>0,\\
        \displaystyle \frac{(\frac{2}{p})^{2}}{\Gamma(\frac{1}{p})^{2}}\int_{0}^{1}\cos(x_{1}t^{\frac{1}{p}})\cos(x_{2}(1-t)^{\frac{1}{p}})
    t^{\frac{1}{p}-1}(1-t)^{\frac{1}{p}-1}dt\qquad\text{if }\omega=0.
      \end{cases}
    \end{equation*}}
  \end{definition}  
  \begin{proposition}[\itshape{Generalizations of Proposition \ref{series} to higher orders}]\label{S-Om}\itshape{}
    \begin{equation*}
      J_{\omega}^{[p]}(x)=\frac{(\frac{|x|_{p}}{p})^{\omega}(\frac{2}{p})^{2}}
      {\Gamma(\frac{1}{p})^{2}}\sum_{k=0}^{\infty}\frac{(-1)^{k}}{\Gamma(\frac{2}{p}(k+1)
      +\omega)}\sum_{m\in\mathbb{N}_{0}^{2}\ |m|'=k}\frac{\Gamma(\frac{2m+1}{p})}{(2m)!}
      x^{2m}\qquad\text{for }\omega\geq0.
    \end{equation*}
  \end{proposition}
  Then, since (\ref{p=2}) and 
    \begin{align}
    J_{\omega}^{[2]}(x)&=\frac{|x|^{\omega}}{2^{\omega-1}\Gamma(\omega)}
      \int_{0}^{1}J_{0}^{[2]}(\tau x)\tau(1-\tau^{2})^{\omega-1}d\tau\notag\\
      &=\frac{|x|^{\omega}}{2^{\omega-1}\Gamma(\omega)}
      \int_{0}^{1}J_{0}(|\tau x|)\tau(1-\tau^{2})^{\omega-1}d\tau\notag\\
      &=\frac{|x|^{\omega}}{2^{\omega-1}\Gamma(\omega)}\sum_{m=0}^{\infty}
      \frac{(-1)^{m}}{m!\Gamma(m+1)}\Bigl(\frac{|x|}{2}\Bigr)^{2m}\int_{0}^{1}
      \tau^{2m+1}(1-\tau^{2})^{\omega-1}d\tau\notag\\
      &=\frac{|x|^{\omega}}{2^{\omega-1}\Gamma(\omega)}\sum_{m=0}^{\infty}
      \frac{(-1)^{m}}{m!\Gamma(m+1)}\Bigl(\frac{|x|}{2}\Bigr)^{2m}
      \frac{B(m+1,\omega)}{2}\notag\\
      &=\sum_{m=0}^{\infty}\frac{(-1)^{m}}{m!\Gamma(m+\omega+1)}
      \Bigl(\frac{|x|}{2}\Bigr)^{2m+\omega}=J_{\omega}(|x|)\label{J_closed}
    \end{align} 
  hold, it follows that $J_{\omega}^{[p]}$ are generalizations of $J_{\omega}$ for non-negative order $\omega$, which preserve the property of spherical symmetry.
  \begin{remark}
  It is clear that Proposition \ref{S-Om} provides a generalization to nonnegative orders of the two-dimensional result of Richards and Castell, namely the series representation of the function corresponding to order zero \cite[Theorem 4.3]{Richards-1986} or \cite[Theorem 4]{Castell}. Furthermore, in our previous work \cite{K3}, we obtained the series representation (Proposition \ref{S-Om}) by defining the function so that it coincides with $J_{\omega}$ in the case $p=2$. In this paper, however, the same series representation is recovered even when $J_{\omega}^{[p]}$ is defined through the above integral relation. Therefore, in view of \cite{K4} and future works, the definition of $J_{\omega}^{[p]}$ can be regarded as a natural one for establishing a theory of new special functions associated with the $p$-circle possessing cusps.
  \end{remark}
  \section{Proof of Theorem \ref{thm}}
  \hspace{13pt}Let $a>0$ and $\beta>-1$. Then, the generalized version of the equality (\ref{HI}) in harmonic analysis, which plays an important role in Kuratsubo and Nakai's paper\cite{Kuratsubo-2022}, can be formulated using the functions
  \begin{equation*}
      U_{\beta,a}(x):=
        \begin{cases}
          (a^{2}-|x|^{2})^{\beta} & \text{if }x\in\mathbb{R}^{2},\ |x|<a,\\
          0 & \text{if }x\in\mathbb{R}^{2},\ |x|\geq a.
        \end{cases}
  \end{equation*}
  We remark that the equality differs from \cite[Theorem 1.2]{K3}, namely, \textit{the Hardy-type oscillatory identity for the astroid-type p-circle domains}.\par
  Let us generalize $U_{\beta,a}$ for $0<p\leq2$ such that $2/p\in\mathbb{N}$ as 
  \begin{equation*}
    U_{\beta,a}^{[p]}(x):=
    \begin{cases}
      (a^{p}-|x|_{p}^{p})^{\beta} & \text{if }x\in\mathbb{R}^{2},\ |x|_{p}<a,\\
      0 & \text{if }x\in\mathbb{R}^{2},\ |x|_{p}\geq a.
    \end{cases}
  \end{equation*}
  Then $U_{\beta,a}^{[p]}$ is $p$-radial, and in particular $U^{[p]}_{0,a}$ is the indicator function on the $p$-circle. We also note that from the assumption on $\beta$ it is easy to check that $U_{\beta,a}^{[p]}$ is integrable on $\mathbb{R}^{2}$, and hence its Fourier transform is given by
    \begin{equation*}
      \hat{U}^{[p]}_{\beta,a}(\xi):=\int_{\mathbb{R}^{2}}U_{\beta,a}^{[p]}(x)e^{-2\pi i\xi\cdot x}dx \qquad\text{for }\xi\in\mathbb{R}^{2}.
    \end{equation*}\par
  Next, we derive a representation of $\mathcal{D}_{\beta}^{[p]}$(see (\ref{Dp}) for definition) by $J_{\omega}^{[p]}$. 
  \begin{proposition}\label{mathD_J}
    \itshape{Let $0<p\leq2$ such that $2/p\in\mathbb{N}$ and $\beta>-1$. Then, the following holds.}
    \begin{equation*}
      \mathcal{D}_{\beta}^{[p]}(s:x)=s^{\beta+\frac{2}{p}}p^{\beta+1}\Gamma(1/p)^{2}\frac{J_{\beta+1}^{[p]}(2\pi\sqrt[p]{s}x)}{(2\pi \sqrt[p]{s}|x|_{p}
      )^{\beta+1}}\qquad\text{for }s>0,\ x\in\mathbb{R}^{2}.
    \end{equation*}
  \end{proposition}
  \begin{proof}
    By display (\ref{prad-Fourier}) and Definition \ref{I-Om}, it follows that
    \begin{align}
      \Gamma(\beta+1)\mathcal{D}_{\beta}^{[p]}(1:x)=\hat{U}_{\beta,1}^{[p]}(x)
      &=p\Gamma(1/p)^{2}\int_{0}^{1}J_{0}^{[p]}(2\pi rx)r(1-r^{p})^{\beta}dr
      \notag\\
      &=p^{\beta+1}\Gamma(\beta+1)\Gamma(1/p)^{2}\frac{J_{\beta+1}^{[p]}
      (2\pi x)}{|2\pi x|_{p}^{\beta+1}}=:\Psi^{[p]}(x)\label{Psi}.
    \end{align}\par
    In addition, let $\Phi^{[p]}:=U_{\beta,1}^{[p]}$ (that is, $\hat{\Phi}^{[p]}=\Psi^{[p]}$), 
    $\Phi_{a}^{[p]}(x):=a^{p\beta}\Phi^{[p]}(x/a)\text{ for }\beta>-1,\ a>0$. Then, we have
    \begin{equation*}
      \Phi_{a}^{[p]}(x)=
      \begin{cases}
        a^{p\beta}(1-|x/a|_{p}^{p})^{\beta}=(a^{p}-|x|_{p}^{p})^{\beta} & 
        (|x|_{p}<a),\\
        0 & (|x|_{p}\geq a),
      \end{cases}
    \end{equation*}
    that is, $\Phi_{a}^{[p]}=U_{\beta,a}^{[p]}$ and $\hat{\Phi}_{a}^{[p]}=\hat{U}_{\beta,a}^{[p]}$. On the other hand, we also have
    \begin{equation*}
      \hat{\Phi}_{a}^{[p]}(x)=a^{p\beta}|1/a|^{-2}\hat{\Phi}^{[p]}(ax)
      =a^{p\beta+2}\Psi^{[p]}(ax)
      =a^{p\beta+2}p^{\beta+1}\Gamma(\beta+1)\Gamma(1/p)^{2}
      \frac{J_{\beta+1}^{[p]}(2\pi ax)}{|2\pi ax|_{p}^{\beta+1}},
    \end{equation*}
    and thus we obtain the desired representation of $\mathcal{D}_{\beta}^{[p]}$ from the above. 
    \begin{align*}
      \mathcal{D}_{\beta}^{[p]}(s:x)&=\frac{1}{\Gamma(\beta+1)}\int_{|\xi|_{p}^{p}
      <s}(s-|\xi|_{p}^{p})^{\beta}e^{2\pi ix\cdot\xi}d\xi=\frac{1}{\Gamma(\beta+1)}
      \hat{U}_{\beta,\sqrt[p]{s}}^{[p]}(x)\\
      &=s^{\beta+\frac{2}{p}}p^{\beta+1}\Gamma(1/p)^{2}
      \frac{J_{\beta+1}^{[p]}(2\pi\sqrt[p]{s}x)}{(2\pi \sqrt[p]{s}|x|_{p})^{\beta+1}}
      \qquad\text{ for }s>0,\ x\in\mathbb{R}^{2}.
    \end{align*}
  \end{proof}
  \begin{remark}
    The representation in Proposition $\ref{mathD_J}$ is a generalization of that (2.5) in \cite{Kuratsubo-2022}
    \begin{equation*}
      \mathcal{D}_{\beta}(s:x)=s^{\beta+1}2^{\beta+1}\pi
      \frac{J_{\beta+1}(2\pi\sqrt{s}|x|)}{(2\pi\sqrt{s}|x|)^{\beta+1}}.
    \end{equation*}\par
    In fact, in the case $p=2$, it is clear from (\ref{J_closed}) that we can express
    \begin{equation*}
      \mathcal{D}_{\beta}^{[2]}(s:x)=s^{\beta+1}2^{\beta+1}\pi
      \frac{J_{\beta+1}^{[2]}(2\pi\sqrt{s}x)}{(2\pi\sqrt{s}|x|)^{\beta+1}}
      =s^{\beta+1}2^{\beta+1}\pi\frac{J_{\beta+1}(2\pi\sqrt{s}|x|)}{(2\pi\sqrt{s}
      |x|)^{\beta+1}}.
    \end{equation*}
  \end{remark}
  \vspace{10pt}\par
  Finally, the proof of the main result is completed by introducing the following formula, which is well known as the periodization of integrable functions.
  \begin{lemma}[\itshape{Poisson summation formula: \cite[Theorem 2.4]{Stein-1971}}]\label{Psf}\itshape{
    Let $F$ be an integrable function on $\mathbb{R}^{d}\ (d\in\mathbb{N})$. Then, the series $f(x):=\sum_{m\in\mathbb{Z}^{d}}F(x+m)$ converges in the $L^{1}$-norm on
$\mathbb{T}^{d}:=(-\frac{1}{2},\frac{1}{2}]^{d}$, and $f$ has the Fourier expansion $\sum_{m\in\mathbb{Z}^{d}}\hat{F}(m)e^{2\pi ix\cdot m}$.}
  \end{lemma}
  
  \begin{proof}[Proof of Theorem \ref{thm}]
  Assume $\beta>-1$ such that $\mathcal{D}_{\beta}^{[p]}(1:x)$ (that is, $\Psi^{[p]}(x)$) is integrable on $\mathbb{R}^{2}$.\par
  
  First, under this assumption, since the associated periodized function is $\mathbb{Z}^{2}$-periodic, it suffices to verify the absolute convergence on $\mathbb{T}^{2}$; that is, the pointwise convergence on $\mathbb{R}^{2}$ follows from the absolute convergence on $\mathbb{T}^{2}$. \par
  In fact, for $n\neq0$ and $x\in\mathbb{T}^{2}$, since
  \begin{equation*}
  |x-n|_{p}^{p}\geq|x_{j}-n_{j}|^{p}\geq(|n_{j}| -|x_{j}|)^{p}\geq\Bigl(1-\frac{1}{2}\Bigr)^{p}=\Bigl(\frac{1}{2}\Bigr)^{p}
  \qquad\text{for }j=1,2, 
  \end{equation*}
  namely, $|x-n|_{p}\geq1/2$ holds, by Proposition \ref{mathD_J} and the integrable assumption,
    \begin{align*}
      \sum_{n\in\mathbb{Z}^{2}\setminus\{0\}}\Bigl|\frac{J_{\beta+1}^{[p]}(2\pi\sqrt[p]{s}(x-n))}{(2\pi\sqrt[p]{s}|x-n|_{p})^{\beta+1}}\Bigr|&\leq\int_{\frac{1}{2}\leq|x-y|_{p}}\frac{|J_{\beta+1}^{[p]}(2\pi\sqrt[p]{s}(x-y))|}{(2\pi\sqrt[p]{s}|x-y|_{p})^{\beta+1}}dy\\
      &=\int_{\frac{\sqrt[p]{s}}{2}\leq|y'|_{p}}\frac{|J_{\beta+1}^{[p]}(2\pi y')|}{(2\pi|y'|_{p})^{\beta+1}}s^{-\frac{2}{p}}dy'\\
      &=C_{\beta,s}^{[p]}\int_{\frac{\sqrt[p]{s}}{2}\leq|y'|_{p}}|\mathcal{D}_{\beta}^{[p]}(1:y')|dy'<+\infty.
    \end{align*}
  holds.\par
  Now, let $\Psi_{a}^{[p]}(x):=a^{2+p\beta}\Psi^{[p]}(ax)$ by (\ref{Psi}). Then, from the Fourier inverse transform, 
  \begin{equation*}
    \hat{\Psi}_{a}^{[p]}(x)=a^{2+p\beta}(1/a^{2})\hat{\Psi}^{[p]}
    (x/a)
    =a^{p\beta}\int_{\mathbb{R}^{2}}\Psi^{[p]}(\xi)e^{-2\pi i\xi\cdot\frac{x}{a}}d\xi
    =a^{p\beta}\int_{\mathbb{R}^{2}}\hat{\Phi}^{[p]}(\xi)e^{-2\pi i\xi\cdot\frac{x}{a}}
    d\xi
    =\Phi_{a}^{[p]}(x)
  \end{equation*}
  holds. Moreover, by using equation 
  \begin{equation*}
    \Gamma(\beta+1)\mathcal{D}_{\beta}^{[p]}(1:s^{\frac{1}{p}}(x-n))
    =\int_{|\xi|_{p}<1}(1-|\xi|_{p}^{p})^{\beta}e^{2\pi is^{\frac{1}{p}}(x-n)\cdot\xi}
    d\xi
    =\int_{|\eta|_{p}^{p}<s}(1-|\eta|_{p}^{p}/s)^{\beta}e^{2\pi i(x-n)\cdot\eta}
    s^{-\frac{2}{p}}d\eta
  \end{equation*}
  together with the integrability assumption and Poisson summation formula (Lemma \ref{Psf}), its periodization $\sum_{n\in\mathbb{Z}^{2}}\Psi_{a}^{[p]}(x-n)$ is integrable on $\mathbb{T}^{2}$, and
  \begin{align*}
    \Gamma(\beta+1)\sum_{n\in\mathbb{Z}^{2}}\mathcal{D}_{\beta}^{[p]}(s:x-n)
    &=\Gamma(\beta+1)\sum_{n\in\mathbb{Z}^{2}}s^{\frac{2}{p}+\beta}
    \mathcal{D}_{\beta}^{[p]}(1:s^{\frac{1}{p}}(x-n))\\
    &=\sum_{n\in\mathbb{Z}^{2}}s^{\frac{2}{p}+\beta}\Psi^{[p]}(s^{\frac{1}{p}}(x-n))\\
    &=\sum_{n\in\mathbb{Z}^{2}}\Psi_{\sqrt[p]{s}}^{[p]}(x-n)\\
    &=\sum_{n\in\mathbb{Z}^{2}}\hat{\Psi}_{\sqrt[p]{s}}^{[p]}(n)e^{2\pi in\cdot x}\\
    &=\sum_{|n|_{p}^{p}<s}(s-|n|_{p}^{p})^{\beta}e^{2\pi in\cdot x}\\
    &=\Gamma(\beta+1)D_{\beta}^{[p]}(s:x)
  \end{align*}
  holds. Combining the above identity with Proposition \ref{mathD_J} gives
  \begin{align*}
   D_{\beta}^{[p]}(s:x)-\mathcal{D}_{\beta}^{[p]}(s:x)
    &=\sum_{n\in\mathbb{Z}^{2}\setminus\{0\}}\mathcal{D}_{\beta}^{[p]}(s:x-n)\\
    &=s^{\beta+\frac{2}{p}}p^{\beta+1}\Gamma(1/p)^{2}
    \sum_{n\in\mathbb{Z}^{2}\setminus\{0\}}\frac{J_{\beta+1}^{[p]}
    (2\pi\sqrt[p]{s}(x-n))}{(2\pi \sqrt[p]{s}|x-n|_{p})^{\beta+1}}
    \quad\text{for }s>0,\ x\in\mathbb{R}^{2}.
  \end{align*}
  \end{proof}
  \section{Concluding remarks (specifying $\beta$ in the assumption of Theorem \ref{thm})}
  \hspace{13pt}
  In this paper, we generalize the series representation (\ref{D-J}) of S. Kuratsubo and E. Nakai \cite{Kuratsubo-2022} to the case such that $2/p\in\mathbb{N}$ (Theorem \ref{thm}). This generalization serves as the first step toward the broader goal of improving the estimates of the $p$-circle lattice points error term for the unsolved cases of $p$. Since Proposition \ref{prop2022} gives an estimate for the error term $P_{2}$
based on the series representation (\ref{D-J}) and the properties of the
Bessel functions established in \cite{Kuratsubo-2022}, it is necessary to investigate
the $p$-Bessel functions $J_{\omega}^{[p]}$ to tackle the problem by
this method. \par
  More specifically, as the next step, we would like to obtain uniform asymptotic evaluations of $J_{\omega}^{[p]}$, using the identity (\ref{Psi}), in order to determine the smallest admissible value of $\beta$ such that the assumption of Theorem \ref{thm} (namely, $\mathcal{D}_{\beta}^{[p]}(1:x)$ is integrable on $\mathbb{R}^{2}$) is satisfied.\par\vspace{5pt}
  For example, in the case $p=2$, the equality (\ref{J_closed}) and the asymptotic formula of the Bessel functions \cite[p199; (1)]{Watson}
  \begin{equation*}\label{asym.}
    J_{\omega}(r)=\sqrt{\frac{2}{\pi r}}\cos\Bigl(r-\frac{2\omega+1}{4}\pi\Bigr)+\mathcal{O}
    (r^{-\frac{3}{2}})
  \end{equation*}
  yield the evaluation formula
  \begin{equation*}
    \frac{J_{\beta+1}^{[2]}(x)}{|x|^{\beta+1}}=\frac{J_{\beta+1}(|x|)}{|x|^{\beta+1}}
    =\mathcal{O}(|x|^{-(\beta+\frac{3}{2})}).
  \end{equation*}
  Then, from the continuity of $J_{\beta+1}(s)/s^{\beta+1}$ (the value at the origin is defined by the limit $1/2^{\beta+1}\Gamma(\beta+2)$), $\beta>1/2$ is given as a sufficient condition for the assumption of (\ref{D-J}).\par
  In general, if the optimal uniform asymptotic exponent $q_{\omega}^{[p]}$ of the $p$-Bessel function of order $\omega\geq0$ on $\mathbb{R}^{2}$ is determined, then the integrability condition in Theorem \ref{thm} can be rewritten as $\beta>1-\inf_{\omega\geq0}q_{\omega}^{[p]}$ (see Section 3 of \cite{K2}).
  
\par\vspace{5pt}
  Therefore, following the method of identifying the infimum of $\beta$ for the special case $p=2$, we need to derive the asymptotic formulas for the generalized Bessel function $J_{\omega}^{[p]}$ for general $p$. Here, for $0<p<1$ such that $2/p$ is an integer greater than or equal to $3$, an asymptotic formula along the axis \cite[Theorem 1.5]{K4}
  \begin{equation*}
    J_{\omega}^{[p]}(x_{\mathrm{axis}})=\frac{4\Gamma(p+1)\sin\frac{\pi p}{2}}{p^{\omega+1}\Gamma(\omega+\frac{1}{p}-1)\Gamma(\frac{1}{p})}r^{\omega-p-1}+\frac{4p^{\frac{1}{p}-2}\cos(r-\frac{\pi}{2}(\frac{1}{p}+\omega))}{\Gamma(\frac{1}{p})}r^{-\frac{1}{p}}+\mathcal{O}(r^{-\gamma_{p,\omega}})\quad \text{as }r\to\infty
  \end{equation*}
has already been derived for general nonnegative orders $\omega$ with $x_{\mathrm{axis}}=(\pm r,0), (0,\pm r)$ and $\gamma_{p,\omega}:=\mathrm{min}\{-\omega+2p+1,1/p+1\}$, and it is known that, unlike the classical Bessel functions, the $p$-Bessel functions exhibit order-dependent asymptotic behavior. Note that $p$-Bessel functions do not necessarily inherit all the properties of classical Bessel functions, as illustrated by the differential formula established in Proposition 2.6 of \cite{K3} and several analytical properties discussed in \cite{K4}, among other examples. \vspace{5pt}\par
  In addition, for the polar coordinate transformation with distorted angles
  \begin{equation*}\label{distorted}
    x_{1}=\mathrm{sgn}(\cos\varphi)r|\cos\varphi|^{\frac{2}{p}},\quad
    x_{2}=\mathrm{sgn}(\sin\varphi)\ r|\sin\varphi|^{\frac{2}{p}},\quad
    \text{for }r\geq0,\ 0\leq\varphi<2\pi,
  \end{equation*}
  we have also obtained uniform asymptotic estimates for distorted angles, but only when $\omega=0$.
 \begin{proposition}[\textit{Asymptotic estimates of the $p$-Bessel function of order zero; \cite[Theorems 1.4-1.5]{K2}}]
    \itshape{Let $p$ satisfy $2/p\in\mathbb{N}\setminus\{1,2\}$. Then, }
    \begin{equation*}
      J_{0}^{[p]}(x)=\mathcal{O}(|x|_{p}^{-\frac{1}{2}}) \qquad\text{as }|x|_{p}\to\infty
    \end{equation*}
holds uniformly with respect to $\varphi$ on any compact subset of $[0,2\pi)\setminus\{0,\frac{\pi}{2},\pi,\frac{3}{2}\pi\}$. Furthermore, the following holds uniformly with respect to $\varphi\in[0,2\pi)$.
    \begin{equation*}
    J_{0}^{[p]}(x)=\mathcal{O}(|x|_{p}^{-\frac{p}{2}})\qquad\text{as }|x|_{p}\to\infty.
    \end{equation*}
\end{proposition}
\vspace{5pt}\par
However, the asymptotic behavior of $J_{\omega}^{[p]}$ along the coordinate axes alone does not determine its uniform asymptotic behavior on $\mathbb{R}^{2}$. In particular, the growth or decay rate obtained along the axes does not necessarily extend uniformly to all directions, and an upper bound alone does not establish the optimal rate. Therefore, a fundamental problem for future research is to establish uniform asymptotic estimates for $J_{\omega}^{[p]}$ of general order $\omega$. 

  \section*{Acknowledgement}
  \hspace{13pt}I would like to express my gratitude to Prof. Mitsuru Sugimoto and Prof. Kohji Matsumoto for their valuable suggestions and helpful remarks on harmonic analysis and number theory. I am also grateful to the anonymous referees for their careful reading of the manuscript and valuable comments, which greatly improved the presentation of the paper.\\
\hspace{13pt}The author is supported by JST SPRING (Grant Number JPMJSP2125) and is currently supported by the JSPS Research Fellowship for Young Scientists (Grant Number JP1337765). The author also thanks the ``THERS Make New Standards Program for the Next Generation Researchers.''
  
  \itshape{
  \hspace{13pt}The author's affiliation: Graduate School of Mathematics, Nagoya University, Chikusa-ku, Nagoya 464-8602, Japan\\
  \hspace{13pt}The author's email address: kitajima.masaya.p2@a.mail.nagoya-u.ac.jp}
\end{document}